\documentclass[12pt, oneside]{article}
\usepackage{amsmath, enumerate, amsthm, amssymb, wasysym, verbatim, bbm, color, graphics, geometry,graphicx,mathrsfs}
\usepackage{algorithmic,algorithm}
\usepackage{listings}
\usepackage{float}
\usepackage{subfigure}
\usepackage{extarrows}
\usepackage[citecolor=blue,colorlinks=true,linkcolor=black]{hyperref}
\newtheorem{thm}{Theorem}[section]
\newtheorem{example}[thm]{Example}
 
\newtheorem{rem}[thm]{Remark}
\newtheorem{lem}[thm]{Lemma}

\newtheorem{prop}[thm]{Proposition}

\newcommand{\mi}{\mathrm{i}}
\newcommand{\md}{\mathrm{d}}
\newcommand{\me}{\mathrm{e}}
\numberwithin{equation}{section}
\numberwithin{figure}{section}

\begin{document}
\title{Monte Carlo estimation of the solution of fractional partial differential equations}
\author{Vassili Kolokoltsov\footnote{also affiliated with HSE Moscow, Russia. email: \texttt{v.kolokoltsov@warwick.ac.uk}} \hspace{5mm}  
Feng Lin\footnote{email: \texttt{feng.lin.1@warwick.ac.uk}} \hspace{5mm}
Aleksandar Mijatovi\'c\footnote{also affiliated with The Alan Turing Institute, UK. email: \texttt{a.mijatovic@warwick.ac.uk}}\\
Department of Statistics, University of Warwick}

\date{}
\maketitle
\begin{abstract}
The paper is devoted to the numerical solutions of fractional PDEs based on its
probabilistic interpretation, that is, we construct approximate solutions via
certain Monte Carlo simulations. The main results represent the upper
bound of errors between the exact solution and the Monte Carlo
approximation, the estimate of the fluctuation via the appropriate
	central limit theorem (CLT) and the construction of confidence intervals.
Moreover, we provide rates of convergence in the CLT via Berry-Esseen type bounds.
Concrete numerical computations and illustrations are included.

\end{abstract}

\section{Introduction}
The study of fractional partial differential equations
(FPDEs) is a very popular topic of modern research due to their ubiquitous
application in natural sciences.  In particular, there is an immense amount of
literature devoted to numerical solution of FPDEs. However most of them exploit
the various kinds of deterministic algorithms (lattice approximation, finite
element methods, etc), see e.g.~\cite{MR3468017,baleanu2012fractional,burrage2017numerical,vabishchevich2016numerical}   and numerous references
therein. However, there are only few papers based on probabilistic methods. For
instance, \cite{uchaikin2003stochastic}  exploits the CTRW (continuous time random walk) approximation for
solutions to FPDEs, and~\cite{lv2020stochastic}  is based on the exact probabilistic
representation.

CTRW approximation to the solutions of FPDEs was developed by physicists more
than half a century ago and it became one of the basic stimulus to the modern
development of fractional calculus. Exact probabilistic representation appeared
a bit later first for fractional equations and then for generalized fractional
(e.g. mixed fractional), see e.g.~\cite{kolokoltsov2009generalized,kolokoltsov2019probabilistic,kochubei2016fractional,meerschaert2011stochastic}  for various
versions of this representation. There are now many books with detailed
presentation of the basics of fractional calculus, see e.g.~\cite{kolokoltsov2019differential, meerschaert2011stochastic,MR1265940}.

The paper is devoted to the numerical solutions of fractional PDEs based on its
probabilistic representation with the main new point being the detailed
discussion of the convergence rates. Namely, the main results represent the
upper bound of errors between the exact solution and the Monte Carlo
approximation, the estimate of the fluctuation via the appropriate central
limit theorem and the construction of confidence intervals. Concrete numerical
computations and illustrations are included.

We denote $C_\infty\left(\mathbb R^d\right):=\{f:\mathbb R^d\to\mathbb R \text{ is continous and vanishes at infinity}\}$.
Let $g\in C_\infty\left(\mathbb R^d\right)$, consider the problem
\begin{equation}
\begin{aligned}
\left(-_tD_a^\alpha+A_x\right)u\left(t,x\right)&=-g\left(x\right),\quad \left(t,x\right)\in\left(a,b\right]\times\mathbb R^d,\\
u\left(a,x\right)&=\phi\left(x\right),\quad	x\in\mathbb R^d,
\end{aligned}
\end{equation}
 where $A_x$ is a generator of a Feller semigroup on $C_\infty\left(\mathbb R^d\right)$ acting on $x$, $\phi\in Dom\left(A_x\right)$, the operator $-_tD_a^\alpha$ is a fractional differential operator of Caputo type of order less than 1 acting on the time variable $t\in[a,b]$ defined by
 $$
 -_tD_a^\alpha f(s):=\int_0^{s-a}\frac{f(s-r)-f(s)}{\Gamma(-\alpha)r^{\alpha+1}}\md r+\frac{f(s)-f(a)}{\Gamma(1-\alpha)(s-a)^\alpha}.
 $$

The solution $u\in C_\infty\left(\left(-\infty,b\right]\times\mathbb R^d\right)$ of the problem (1.1) exsits and is given by \cite{hernandez2017generalised}. $u$ has the stochastic representation (see \cite{hernandez2017generalised} equation (4) and Theorem 4.20)
\begin{equation}
u\left(t,x\right)=\mathbb E\left[\phi\left(X^x_{T_t}\right)+\int_0^{T_t}g\left(X^x_s\right)ds\right],
\end{equation}
where $\{X^x_s\}_{s\geqslant0}$ is the stochastic process started at $x\in\mathbb R^d$ generated by $A_x$. Let $\{\tau_s\}_{s\geqslant0}$ be $\alpha$-stable subordinator with $\tau_1$ satisfying $\mathbb E[\me^{\mi z\tau_1}]=\me^{\int_0^\infty(\me^{izx}-1)\frac{\alpha}{\Gamma(1-\alpha)}x^{-\alpha-1}\md x}$ and $T_t:=\inf\{s>0, t-\tau_s<a\}$. 

When $\{X^x_s\}_{s\geqslant0}$ is Brownian motion,  then $A_x$ would be $\frac{1}{2}\Delta$, where $\Delta=\sum_{i=1}^d\left(\frac{\partial}{\partial x_i}\right)^2$. If $\{\tau_s\}$ is the deterministic drift, i.e. $-_tD_a=-\frac{d}{dt}$ and $g=0$, then (1.1) becomes
\begin{equation}
\frac{1}{2}\Delta u\left(t,x\right)=\frac{d}{dt}u\left(t,x\right),
\end{equation}
the heat equation that we are more familiar with.\\

 We assume $\{X^x_s\}_{s\geqslant0}$ is isotropic $\beta$-stable.(What `isotropic' means is explained in Section 2, after Lemma 2.3.) In this paper we shall investigate some properties of the representation (1.2) and its Monte-Carlo estimator, i.e.
 \begin{equation}
u_N^h\left(t,x\right)=\frac{1}{N}\sum_{k=1}^N\left(\phi\left(X_{T_t^k}^{x,k}\right)+\sum_{i=1}^{\lfloor T_t^k/h\rfloor}hg\left(X_{t_i^k}^{x,k}\right)\right),
\end{equation}
where $h>0$ is the step length, $T_t^k$ are iid samples of $T_t$, and $t_i^k=\left(i-1\right)h$. Note that we can sample the stopping time $T_t$(see Lemma 2.4 below), then sample the isotropic $\beta$-stable process $\{X_s^x\}$ and finally simulate the estimator (1.4).\\

	In Section 2 we mainly focus on the situation when $g=0$, i.e. the estimator now is
	\begin{equation}
	u_N\left(t,x\right)=\frac{1}{N}\sum_{k=1}^N\phi\left(X_{T_t^k}^{x,k}\right).
	\end{equation}
	To make central limit theorem and Berry-Esseen bound hold, we only need to estimate the tail of the stable process at some stopping time, i.e. $\mathbb P\left[|X_{T_t}^x|>s\right]$ for large $s$. And we begin with showing that the order of the tail of multidimentional stable distribution has the same order of the tail of each component of itself. In Section 3 we study the property of the Monte-Carlo estimator when the forcing term $g\neq0$. We estimate the upper bound of the second moment of the estimator and then, the $L^2$ error between the estimator and the solution. Besides, we use there properties to show that the central limit theorem holds using the triangular arrays. In Section 4 we give numerical examples, demonstrating the performance of our simulation algorithm.

\section{Properties of the estimator when the forcing term g=0}
In this paper, for function $f,g:\mathbb R^d\to\mathbb R$, we use notation $f\left(x\right)=O\left(g\left(x\right)\right)$, meaning that $|\frac{f\left(x\right)}{g\left(x\right)}|$ is bounded as $|x|\to\infty$. Also we use the notation $f(x)\sim g(x)$, meaning that both $|\frac{f(x)}{g(x)}|$ and $|\frac{g(x)}{f(x)}|$ are bounded as $|x|\to\infty$.\\

In this section, we study the situation when $g\left(x\right)=0$ for all $x\in \mathbb R^d$, then the stochastic representation (1.2) becomes
\begin{equation}
u\left(t,x\right)=\mathbb E\left[\phi\left(X^x_{T_t}\right)\right]
\end{equation}
and the estimator now is defined in (1.5).\\

Our main results tell us how close $u_N\left(t,x\right)$ and $u\left(t,x\right)$ are:
\begin{thm}
\begin{enumerate}[(i)]
	\item
For all continuous function $\phi:\mathbb R^d\to\mathbb R$,
\begin{equation}
u_N\left(t,x\right)\overset{a.s.}\to u\left(t,x\right), \text{ as } N\to\infty.
\end{equation}
\item
Let $S_N\left(t,x\right)=\sqrt N\left(u_N\left(t,x\right)-u\left(t,x\right)\right)/\sigma\left(t,x\right)$ and $W$ be the standard normal distribution.
If $\phi\left(x\right)$ satisfies $\phi\left(x\right)=O\left(|x|^{\frac{\beta}{2+\delta}}\right)$, where $\delta>0$, then the central limit theorem holds, i.e. for all bounded uniformly continuous funtion $\psi$,
$$
\mathbb E\left[\psi\left(S_N\left(t,x\right)\right)\right]\to\mathbb E\left[\psi\left(W\right)\right] \text{ as } N\to\infty.
$$

\item
Let $Y\left(t,x\right):=\phi\left(X_{T_t}^x\right)-\mathbb E\left[\phi\left(X_{T_t}^x\right)\right]$, denote $\mathbb E\left[Y\left(t,x\right)^2\right]=\sigma\left(t,x\right)^2$, $\mathbb E\left[|Y\left(t,x\right)|^3\right]=\rho\left(t,x\right)$. If $\phi\left(x\right)$ satisfies $\phi\left(x\right)=O\left(|x|^{\frac{\beta}{3+\delta}}\right)$, where $\delta>0$, then for all $C^3$ functions $\psi:\mathbb R\to\mathbb R$,
$$
|\mathbb E\left[\psi\left(S_N\left(t,x\right)\right)\right]-\mathbb E\left[\psi\left(W\right)\right]|\leqslant0.433||\psi'''||_\infty\frac{\rho\left(t,x\right)}{\sqrt N\sigma\left(t,x\right)^3},
$$

Here $C^3$ means the space of functions with bounded third derivatives.
\end{enumerate}
\end{thm}

In other words, the central limit theorem can be written using convergence in distribution:
$$
\sqrt N\left(u_N\left(t,x\right)-u\left(t,x\right)\right)\overset{d}\to N\left(0,\sigma\left(t,x\right)^2\right) \text{ as }N\to\infty.
$$

Since the estimator is unbiased, Theorem 2.1(i) holds because of the strong law of large numbers. For (ii), it is the standard central limit theorem and we only need to show that $\mathbb E\left[\phi\left(X_{T_t}^x\right)^2\right]<\infty$. For (iii), it is a version of the Berry-Esseen bound and we need to show that $\mathbb E\left[|\phi\left(X_{T_t}^x\right)|^3\right]<\infty$. These facts are evident if $\phi\left(x\right)$ is bounded. To deal with unbounded $\phi\left(x\right)$, let us recall the following fact: for any random variable $U$,
\begin{equation}
\mathbb E\left[U^2\right]=\int_0^\infty\mathbb P\left[U^2>t\right]dt.
\end{equation}
It is finite if $\mathbb P\left[|U|>t\right]=O\left(t^{-\left(2+\delta\right)}\right)$, where $\delta$ is a positive constant. Now let us look back at our problems. Once we know the tail of $X_{T_t}^x$ and the growth rate of $\phi\left(x\right)$, the tail of $\phi\left(X_{T_t}^x\right)$ would be clear as well as the finiteness of the moments of $\phi\left(X_{T_t}^x\right)$.\\

Luckily, we have following result:
\begin{prop}
Assume that $\{X_s\}_{s\geqslant0}$ is a $\beta$ stable process, then $\mathbb P\left[|X_{T_t}^x|>u\right]=O\left(u^{-\beta}\right)$.
\end{prop}

To prove Proposition 2.2, we need a little lemma telling us that the distribution of $T_t$ is analytically accessible:
\begin{lem}Denote $\bar a:=t-a$, then $T_t\overset{d}=\left(\frac{\bar a}{\tau_1}\right)^\alpha$. 
\end{lem}
\begin{proof}
Note that $\tau_s\overset{d}=s^{1/\alpha}\tau_1$.
$\{T_t>s\}=\{t-\tau_s>a\}$, since $\tau_s$ has monotone paths. Hence 
$$
\mathbb P\left[T_t>s\right]=\mathbb P\left[t-\tau_s>a\right]=\mathbb P\left[s^{\frac{1}{\alpha}}\tau_1<\bar a\right]=\mathbb P\left[s<\left(\bar a/\tau_1\right)^\alpha\right]
$$
\end{proof}

Together with the facts that $X_s^x$ is $\beta$ stable and Lemma 2.3,
 \begin{equation}
 X_{T_t}^x-x\overset{d}=T_t^{\frac{1}{\beta}}X_1\overset{d}=\left(\frac{\bar a}{\tau_1}\right)^{\frac{\alpha}{\beta}}X_1.
 \end{equation}

Also we need Lemma 2.4 and Lemma 2.5 given below. Now let us explain what `isotropic' means in our assumption of $\{X_s\}_{s\geqslant0}$.\\

For $d$-dim $\beta$-stable random variable $U=\left(U_{\left(1\right)},...,U_{\left(d\right)}\right)$ on $\mathbb R^d$, there are a finte measure $\lambda$ on sphere $S$ and $\gamma$ in $\mathbb R^d$ such that the characteristic function of $U$ satisfies
$$
\hat{U}\left(z\right):=\mathbb E[\me^{i\langle z,U\rangle}]=\exp\left[-\int_S|\langle z,\xi\rangle|^\beta\left(1-\mi\tan\frac{\pi\beta}{2}\mathrm{sgn}\langle z,\xi\rangle\right)\lambda\left(\md\xi\right)+\mi\langle\gamma,z\rangle\right] \text{ for }\beta\neq1,
$$
and vice versa. Hence each component of $U$ is 1-dim stable random variable and the stability index is still $\beta$. Besides, for $1$-dim $\beta$-stable random variable $V$ whose characteristic function has form
$$
\hat V(z)=\mathbb E[\me^{iVz}]=\exp(-\sigma^\beta|z|(1-i\rho(\text{sign} z)\tan(\pi\beta/2)+i\mu z),
$$ we use the notation $V\sim S_\beta(\sigma,\rho,\mu)$. We say a $d$-dim stable random variable $U$ is \textbf{isotropic} if its coordinates have the same distribution, i.e. $U_{(i)}\sim S_\beta(\sigma,\rho,\mu)$ $i=1,...,d$. We say a process $\{X_{s}\}_{s\geqslant0}$ is isotropic stable if $X_1$ is a isotropic stable random variable.
\begin{lem}
Let $U=(U_{(1)},...,U_{(d)})$ be an isotropic $d$-dim $\beta$-stable random variable, and $U_{\left(i\right)}\sim S_\beta\left(\sigma,\rho,\mu\right)$, then $\mathbb P\left[|U|>s\right]\sim s^{-\beta}$ as $s\to\infty$.
\end{lem}
\begin{lem}
Let $U,V$ be positive random variables such that
$$
\lim_{t\to\infty}t^\alpha\mathbb P[U>t]\geqslant C_1,
\lim_{t\to\infty}t^\alpha\mathbb P[V>t]\leqslant C_2,
$$
where $C_1>C_2$, then
$$
\mathbb P[U-V>t]=O\left(t^{-\alpha}\right) \text{ for }t\to\infty.
$$
\end{lem}
Lemma 2.4. tells us the order of tail of high dimentional stable process. Lemma 2.5. shows the order of the difference between certian random variables and we can apply it to the logarithm of (2.4), i.e. $\log|X_1|+\frac{\alpha}{\beta}\log\bar a-\frac{\alpha}{\beta}\log\tau_1$.
\begin{proof}[\textbf{Proof of Lemma 2.4}]
Since $\{|U|=\sqrt{U_{\left(1\right)}^2+...+U_{\left(d\right)}^2}>s\}\supset\{|U_{\left(1\right)}|>s\}$, we have
$$
\mathbb P[|U|>s]\geqslant\mathbb P[|U_{\left(1\right)}|>s].
$$
Since $\{|U|>s\}\subset\{\max_{1\leqslant i\leqslant d}|U_{\left(i\right)}|>s/\sqrt{d}\}\subset\cup_{i=1}^d\{|U_{\left(i\right)}|>s/\sqrt{d}\}$, we have
$$
\mathbb P[|U|>s]\leqslant\sum_{i=1}^d\mathbb P[U_{\left(i\right)}>s/\sqrt{d}].
$$
Now recall the well known result of the tail of $1$-dim stable random variable: if $V\sim S_\beta(\sigma,\rho,\mu)$, then
\begin{equation}
\lim_{s\to\infty}s^\beta\mathbb P\left[|V|>s\right]=C_\beta\sigma^\beta,
\end{equation}
where $C_\beta=\left(\int_0^\infty x^{-\beta}\sin xdx\right)^{-1}=\frac{1-\beta}{\Gamma\left(2-\beta\right)\cos\left(\pi\beta/2\right)}$(see \cite{samorodnitsky1994stable}, property 1.2.15).\\
Hence for any $\epsilon>0$, there exists some $M$, such that for all $s>M$ and $i=1,...,d$,
$$
(C_\beta\sigma^\beta-\epsilon)s^{-\beta}\leqslant\mathbb P[|U_{(i)}|>s]\leqslant \left(C_\beta\sigma^\beta+\epsilon\right)s^{-\beta}.
$$
Hence for $s>\sqrt dM$,
\begin{equation}
\mathbb P[|U|>s]\leqslant\sum_{i=1}^d\mathbb P[|U_{(i)}|>s/\sqrt d]\leqslant d^{1+\beta/2}\left(\epsilon+C_\beta\sigma^\beta\right)s^{-\beta}.
\end{equation}
Therefore $\mathbb P[|X|>s]\sim s^{-\beta}$ as $s\to\infty$.
\end{proof}
\begin{proof}[\textbf{Proof of Lemma 2.5}]
Given a positive number $M$, there exsits $T$ and $\epsilon>0$, such that for all $t>T$,
\begin{equation*}
\begin{aligned}
\mathbb P[U-V>t]&\geqslant\mathbb P[U>\left(M+1\right)t]-\mathbb P[V>Mt]\\
&\geqslant \frac{C_1+\epsilon}{\left(M+1\right)^\alpha}t^{-\alpha}-\frac{C_2-\epsilon}{M^\alpha}t^{-\alpha}\\
&\geqslant\frac{1}{M^\alpha}\left(\left(C_1+\epsilon\right)\left(\frac{M}{M+1}\right)^\alpha-\left(C_2-\epsilon\right)\right)t^{-\alpha}.
\end{aligned}
\end{equation*}
If we pick $M$ big enough, we have $\mathbb P[U-V>t]\geqslant Ct^{-\alpha}$for some constant $C$.\\
On the other hand, for large $t$\\
\begin{equation*}
\begin{aligned}
\mathbb P[U-V>t]&=\int_{V>0}\mathbb P[U-v>t]\mathbb P[V\in\md v]\\
&\leqslant \int_{V>0}\mathbb P[U>t]\mathbb P[V\in dv]\\
&\leqslant \int C_1t^{-\alpha}\mathbb P[V\in\md v]\\
&\leqslant C_1t^{-\alpha}.
\end{aligned}
\end{equation*}
Therefore $\mathbb P[U-V>t]\sim t^{-\alpha}$ as $s\to\infty$.
\end{proof}

\begin{proof}[\textbf{Proof of Proposition 2.2}]
Now let us estimate the tail of $X_{T_t}^x$. For large $u>0$,
\begin{equation}
\begin{aligned}
\mathbb P\left[|X_{T_t}^x|>u\right]=&\mathbb P\left[|\left(\frac{\bar a}{\tau_1}\right)^{\frac{\alpha}{\beta}}X_1+x|>u\right]\leqslant \mathbb P\left[\left(\frac{\bar a}{\tau_1}\right)^{\frac{\alpha}{\beta}}|X_1|>u-|x|\right]\\
=&\mathbb P\left[\log|X_1|-\frac{\alpha}{\beta}\log \tau_1>\log\left(u-|x|\right)-\frac{\alpha}{\beta}\log\bar a\right]\\
=&\mathbb P[A-B>r,A>0,B>0]+\mathbb P[A-B>r,A>0,B<0]+\\
&\mathbb P[A-B>r,A<0,B<0],
\end{aligned}
\end{equation}
where $A:=\log|X_1|$, $B:=\frac{\alpha}{\beta}\log\left(\tau_1\right)$, $r:=\log\left(u-|x|\right)-\frac{\alpha}{\beta}\log\bar a$. (Note that for large $u$ we have $r>0$).\\
Let $X_1=(X_{(1)},...,X_{(d)})$ and $X_{(i)}\sim S_\beta(\sigma,\rho,\mu)$, $i=1,...,d$. By the Proof of Lemma 2.4, for any $\epsilon>0$, there exists some $M$, such that for all $s>M$ and $i=1,...,d$,
$$
\mathbb P[|X_{(i)}|>s]\leqslant \left(C_\beta\sigma^\beta+\epsilon\right)s^{-\beta}.
$$
Hence for $s>\sqrt dM$,
$$
\mathbb P[|X_1|>s]\leqslant\sum_{i=1}^d\mathbb P[|X_{(i)}|>s/\sqrt d]\leqslant d^{1+\beta/2}\left(\epsilon+C_\beta\sigma^\beta\right)s^{-\beta},
$$
and for $t>\log(\sqrt dM)$,
$$
\mathbb P[\log|X_1|>t]=\mathbb P[|X_1|>\me^t]\leqslant d^{1+\beta/2}\left(\epsilon+C_\beta\sigma^\beta\right)\me^{-\beta t}.
$$
Now let us discuss (2.7) in three conditions. For $r>\log(\sqrt dM)$,\\
(1) When $A>0,B>0$, we have
\begin{equation}
\begin{aligned}
\mathbb P[A-B>r, A>0,B>0]&\leqslant\mathbb P[A>r]=\mathbb P[|X_1|>\me^s]\\
&\leqslant d^{1+\beta/2}\left(\epsilon+C_\beta\sigma^\beta\right)\me^{-\beta r}.
\end{aligned}
\end{equation}
(2) When $A>0$, $B<0$, pick integer $k=\lfloor r/S\rfloor$, and we divide the event $\{A+\left(-B\right)>r\}$ into $k$ parts:
\begin{equation}
\begin{aligned}
\{A+\left(-B\right)>r\}&=\bigcup_{i=1}^{k-1}\{A+\left(-B\right)>r,-B\in\left(\frac{i-1}{k}r,\frac{i}{k}r\right]\}\bigcup\{A+\left(-B\right)>r,-B>\frac{k-1}{k}r\}\\
&\subset\bigcup_{i=1}^{k-1}\{A>\frac{k-i}{k}r,-B\in\left(\frac{i-1}{k}r,\frac{i}{k}r\right]\}\bigcup\{-B>\frac{k-1}{k}r\}\\
&\subset\bigcup_{i=1}^k\{A>\frac{k-i}{k}r,-B>\frac{i-1}{k}r\}.
\end{aligned}
\end{equation}
Hence
\begin{equation}
\begin{aligned}
\mathbb P[A+\left(-B\right)>r,A>0,B<0]
&\leqslant\sum_{i=1}^k\mathbb P[A>\frac{k-i}{k}r,-B>\frac{i-1}{k}r]\\
&=\sum_{i=1}^{k}\mathbb P[A>\frac{k-i}{k}r]\mathbb P[-B>\frac{i-1}{k}r].
\end{aligned}
\end{equation}
Recall that
\begin{equation}
\mathbb P[\log|X_1|>\frac{k-i}{k}r]\leqslant d^{1+\beta/2}\left(\epsilon+C_\beta\sigma^\beta\right)\me^{-\frac{k-i}{k}\beta r}.
\end{equation}

Use the result (3.7) that we shall metion later, we have
\begin{equation}
\mathbb E[\tau_1^{-2\alpha}]=\frac{2}{\Gamma\left(1+2\alpha\right)}.
\end{equation}
By Markov inequality,
\begin{equation}
\begin{aligned}
\mathbb P[\frac{\alpha}{\beta}\log\left(\tau_1^{-1}\right)\geqslant \frac{i-1}{k}r]
&=\mathbb P[\tau_1^{-1}>\me^{\frac{\beta}{\alpha}\frac{i-1}{k}r}]\leqslant \frac{\mathbb E[\tau_1^{-2\alpha}]}{\left(\me^{\frac{\beta}{\alpha}\frac{i-1}{k}r}\right)^{2\alpha}}\leqslant2\me^{-2\frac{i-1}{k}\beta r}.
\end{aligned}
\end{equation}
Combining (2.10),(2.11) and (2.13), we have
\begin{equation}
\begin{aligned}
\mathbb P[A+\left(-B\right)>r,A>0,B<0]&\leqslant\sum_{i=1}^{k}d^{1+\beta/2}\left(\epsilon+C_\beta\sigma^\beta\right)\me^{-\frac{k-i}{k}\beta r}2\me^{-2\frac{i-1}{k}\beta r}\\
&=2d^{1+\beta/2}\left(\epsilon+C_\beta\sigma^\beta\right) \sum_{i=1}^{k}\me^{-\frac{i-2}{k}\beta r}\me^{-\beta r}\\
&\leqslant 2d^{1+\beta/2}\left(\epsilon+C_\beta\sigma^\beta\right) \frac{\me^{\beta r/k}}{1-\me^{-\beta r/k}}\me^{-\beta r}\\
&\leqslant 2d^{1+\beta/2}\left(\epsilon+C_\beta\sigma^\beta\right)\frac{\me^{2\beta S}}{\me^{\beta S}-1}\me^{-\beta r}.
\end{aligned}
\end{equation}
(3) When $A<0$, $B<0$, then
\begin{equation}
\begin{aligned}
\mathbb P[A-B>r, A<0, B<0]&\leqslant\mathbb P[A<0,-B>r]\leqslant\mathbb P[-B>r]\\
&=\mathbb P[\frac{\alpha}{\beta}\log\left(\tau_1^{-1}\right)>r]=\mathbb P[\tau_1^{-1}>\me^{\frac{\beta}{\alpha}r}]\\
&\leqslant\frac{\mathbb E[\tau_1^{-\alpha}]}{\left(\me^{\frac{\beta}{\alpha}r}\right)^{\alpha}}\leqslant \me^{-\beta r}.
\end{aligned}
\end{equation}
Combining (1)(2)(3) we know that for large $u$,
\begin{equation}
\mathbb P[|X_{T_t}^x|>u]\leqslant\left(\left(1+2\frac{\me^{2\beta S}}{\me^{\beta S}-1}\right)d^{1+\beta/2}\left(\epsilon+C_\beta\sigma^\beta\right)+1\right)\me^{-\beta r}=O\left(\me^{-\beta\left(\log\left(u-|x|\right)-\frac{\alpha}{\beta}\log\bar a\right)}\right)=O\left(u^{-\beta}\right).
\end{equation}
\end{proof}

Now let us finish the proof of our main result.
\begin{proof}[\textbf{Proof of Theorem 2.1.(ii)}]
With Proposition 2.2, it is easy to see that
\begin{equation}
\mathbb P[\phi\left(X_{T_t}^x\right)>u]=O\left(u^{-\left(2+\delta\right)}\right),
\end{equation}
and by (2.3) $\mathbb E\left[\left(\phi\left(X_{T_t}^x\right)\right)^2\right]$ is finite.
\end{proof}
For the proof of Theorem 2.1.(iii), $\mathbb E[|\phi\left(X_{T_t}^x\right)|^3]$ is finite because of the similar argument. For the rest proof, see \cite{o2014analysis} page 356 Variant Berry-Esseen Theorem.
\begin{rem}
\begin{enumerate}
\item
If $C_\alpha<C_\beta\sigma^\beta$, by Lemma 2.5, we have
\begin{equation}
\mathbb P[A-B>r]\geqslant\mathbb P[A-B>r,A>0,B>0]\geqslant Ct^{-\beta},
\end{equation}
where $C$ is a constant that can be chosen from the proof of Lemma 2.5. This result means the order $t^{-\beta}$ is the best one.
\item In the Proof of Proposition 2.2 we need  $r=\log(u-|x|)-\frac{\alpha}{\beta}\log\bar a$ and $r>\log(\sqrt dM)$. Hence there exists some constant $M_0$  such that for $u>M_0$, (2.16) holds and $M_0$ has order $d^{1/2}$.
\end{enumerate}
\end{rem}

Besides, we can roughly give the upper bound of $\mathbb E[\phi\left(X_{T_t}^x\right)^2]$.
\begin{example}
If $\phi\left(x\right)$ satisfies $\phi\left(x\right)\leqslant |x|^{\frac{\beta}{\delta+2}}$, where $\delta>0$, then from Remark 2.6.2 we know that there exists some $M_0$ such that for all $t>M_0$,
\begin{equation}
\begin{aligned}
\mathbb P[|X_{T_t}^x|>t]&\leqslant\left(\left(1+\frac{\me^{2\beta S}}{\me^{\beta S}-1}\right)d^{1+\beta/2}\left(\epsilon+C_\beta\sigma^\beta\right)+1\right)\me^{-\beta r}=M^{\left(1\right)}\me^{-\beta r}\\
&=M^{\left(1\right)}\me^{-\beta\log\left(t-|x|\right)-\frac{\alpha}{\beta}\log\bar a}\leqslant M^{\left(2\right)}t^{-\beta},
\end{aligned}
\end{equation}
where $M^{\left(1\right)}=\left(1+\frac{\me^{2\beta S}}{\me^{\beta S}-1}\right)d^{1+\beta/2}\left(\epsilon+C_\beta\sigma^\beta\right)+1$, $M^{\left(2\right)}=2\bar a^{-\frac{\alpha}{\beta}}M^{\left(1\right)}$.
Hence
\begin{equation}
\begin{aligned}
\mathbb E[\phi\left(X_{T_t}^x\right)^2]&=\int_0^\infty\mathbb P[\phi\left(X_{T_t}^x\right)^2>t]dt\leqslant M_0+\int_{M_0}^\infty\mathbb P[|X_{T_t}^x|>\sqrt{t}^{\frac{2+\delta}{\beta}}]dt\\
&\leqslant M_0+M^{\left(2\right)}\int_{M_0}^\infty \sqrt{t}^{\frac{2+\delta}{\beta}\cdot\left(-\beta\right)}dt=M_0+2M^{\left(2\right)}M_0^{-\delta/2}/\delta.
\end{aligned}
\end{equation}
Note that $M_0$ has order $d^{1/2}$ and $M^{(2)}$ has order $d^{1+\beta/2}$,
This upper bound has order $d^{1+\beta/2}$.
\end{example}

\section{Properties of the estimator when  $g \neq 0$}
In this section we want to clarify the Monte-Carlo estimator of the stochastic representation in section 1. Here we assume that $g$ satisfies the condition $|g\left(x\right)-g\left(y\right)|\leqslant L|x-y|_\gamma$,  where $|x|_\gamma=\sum_{i=1}^d|x_{(i)}|^\gamma$, $x_{(i)}$ is the coordinate of $x$, $0<\gamma<\beta/2$. \\

Our main results in this section is following:
\begin{thm}
\begin{enumerate}[(i)]
	Assume $|\phi(x)|=O(|x|^{\frac{\beta}{2+\delta}})$ for $|x|\to\infty$, where $\delta>0$. 
	\item Then
	$
	\mathbb E[\left(u_N^h\left(t,x\right)-u\left(t,x\right)\right)^2]\to0$ as $N\to\infty$, $h\to0$.
	\item
	(CLT with a bias correction) Let $h_N=N^{-\frac{2\beta}{\gamma}}$, $u\left(t,x\right)=\mathbb EZ\left(t,x\right)$ where
$$
Z\left(t,x\right)=\phi\left(X_{T_t}^x\right)+\int_0^{T_t}g\left(X_s^x\right)ds,
$$  and $W$ be the standard normal distribution, then for all bounded uniformly continuous function $\psi$,
$$
\mathbb E\left[\psi\left(\sqrt N\left(u_N^{h_N}\left(t,x\right)-u\left(t,x\right)\right)/\sqrt{VarZ\left(t,x\right)}\right)\right]\to\mathbb E[\psi\left(W\right)] \text{ as } N\to\infty.
$$
\end{enumerate}
\end{thm}
\par
Let
$$
Y_h\left(t,x\right)=\phi\left(X_{T_t}^x\right)+\sum_{i=1}^{\lfloor T_t/h\rfloor}hg\left(X_{t_i}^x\right)$$ be the approximation of $Z(t,x)$. And let$$
u_N^h\left(t,x\right)=\frac{1}{N}\sum_{k=1}^NY_h^k(t,x)
$$ where
$Y^k_h(t,x)=\phi\left(X_{T_t^k}^{x,k}\right)+\sum_{i=1}^{\lfloor T_t^k/h\rfloor}hg\left(X_{t_i^k}^{x,k}\right)$, $k=1,...,N$.
$Y_h^k(t,x)$ are the iid copies of $Y_h(t,x)$.
Note that for random variable $U$, let $V$ be its approximation and $V^k$, $k=1,...,N$ be the iid copies of $V$. The $L^2$ error satisfies
\begin{equation}
\mathbb E\left[\left(\mathbb E U-\frac{1}{N}\sum_{k=1}^NV^k\right)^2\right]=\frac{1}{N}var V+\left(\mathbb EU-\mathbb E V\right)^2.
\end{equation}
Therefore, to estimate the $L^2$ error $\mathbb E[\left(u\left(t,x\right)-u_N\left(t,x\right)\right)^2]$, we only need to study $varY_h(t,x)$ and $\mathbb EZ(t,x)-\mathbb EY_h(t,x)$, and the following propositions answer these questions.

\begin{prop}
There exists a constant $M_{t,x}^1$ (depending on $t,x$) such that $VarY_h\left(t,x\right)\leqslant M_{t,x}^1$.
\end{prop}

\begin{prop}
There exists a constant $M_{t,x}^2$ (depending on $t,x$) such that $\mathbb E[|Z\left(t,x\right)-Y_h\left(t,x\right)|]\leqslant M_{t,x}^2h^{\frac{\gamma}{\beta}}$.
\end{prop}
\begin{prop}
There exists a constant $M_{t,x}^3$ (depending on $t,x$) such that $\mathbb E[|Z\left(t,x\right)-Y_h\left(t,x\right)|^2]\leqslant M_{t,x}^3h^{\frac{2\gamma}{\beta}}$.
\end{prop}

Sections 3.1 and 3.2 give proofs of these propositions. Section 3.3 is the proof of our CLT.
\begin{rem}
\begin{itemize}
	\item
	Non-asymptotic confidence interval:
Combining (3.1), proposition 3.2 and proposition 3.3 we have
\begin{equation}
\mathbb E[\left(u\left(t,x\right)-u_N^h\left(t,x\right)\right)^2]\leqslant \frac{1}{N}M_{t,x}^1+\left(M_{t,x}^2\right)^2h^{\frac{2\gamma}{\beta}}
\end{equation}
where $u\left(t,x\right)$ is the solution of problem (1.1). Now we can construct the confidence interval using Markov inequality:
\begin{equation}
\mathbb P[|u\left(t,x\right)-u_N^h\left(t,x\right)|>r]\leqslant \mathbb E[\left(u\left(t,x\right)-u_N^h\left(t,x\right)\right)^2]/r^2\leqslant \frac{1}{r^2}\left(\frac{1}{N}M_{t,x}^1+\left(M_{t,x}^2\right)^2h^{\frac{2\gamma}{\beta}}\right)
\end{equation}
Hence we can pick suitable $N$ and $h$ such that $\mathbb P[|u\left(t,x\right)-u_N^h\left(t,x\right)|>r]<1-\epsilon$ for some small $\epsilon$.
\item
Asymptotic confidence interval: We can use CLT in theorem 3.1. to get the asymptotic confidence interval. In other words, the central limit theorem can be written using convergence in distribution:
\begin{equation}
\sqrt N\left(u_N^{h_N}\left(t,x\right)-u\left(t,x\right)\right)\overset{d}\to N\left(0,VarZ\left(t,x\right)\right)\quad\text{as }n\to\infty.
\end{equation}
Once we have the upper bound $M\left(t,x\right)$ of $\sqrt{VarZ\left(t,x\right)}$(e.g. see Example 2.7.), it is easy to see that it yields a $100\left(1-\alpha\right)\%$ asymptomic confidence interval $u_N^{h_N}\pm\frac{M\left(t,x\right)}{\sqrt N}z\left(\alpha/2\right)$ for $u\left(t,x\right)$, where $z\left(t\right)$ satisfies $\Phi\left(z\left(t\right)\right)=1-t$ and $\Phi$ is the distribution function of the standard normal distribution. See section 4.3 for a simple example.
\end{itemize}
\end{rem}

Before the calculation, we need the following results (see \cite{ken1999levy}. page 162) :\\
(1) For constants $c>0$, $\eta \in (-1,\beta)$ and a symmetric $\beta$-stable $1$-dim process $U_t$ with $\mathbb E[e^{izU_t}]=e^{-tc|z|^\beta}$, we have
\begin{equation}
\mathbb E[|U_t|^\eta]=\left(tc\right)^{\eta/\beta}\frac{2^\eta\Gamma\left(\frac{1+\eta}{2}\right)
\Gamma\left(1-\frac{\eta}{\beta}\right)}{\sqrt{\pi}\Gamma\left(1-\frac{\eta}{2}\right)}.
\end{equation}
Recall that each component of $X_1-X_0$, denoted by $X_{(j)}$, is symmetric with $\mathbb E[e^{izX_{(j)}}]=e^{-c|z|^\beta}$ and $c>0$.\\
(2) If $0<\alpha<1$ and $\{X_t\}$ is a stable subordinator with $\mathbb E[e^{-uX_t}]=e^{-tc'u^\alpha}$, where $c'$ is some constant, then for $-\infty<\eta<\alpha$,
\begin{equation}
\mathbb E[X_t^\eta]=\left(tc'\right)^{\eta/\alpha}\frac{\Gamma\left(1-\frac{\eta}{\alpha}\right)}{\Gamma\left(1-\eta\right)}.
\end{equation}
Since we have $\mathbb E[e^{-u\tau_1}]=\me^{-u^\alpha}$(see \cite{ken1999levy}, Example 24.12),
\begin{equation}
	\mathbb E[\tau_1^\eta]=\frac{\Gamma\left(1-\frac{\eta}{\alpha}\right)}{\Gamma\left(1-\eta\right)}.
\end{equation}
\\
\subsection{Estimation of $VarY_h$}
In this section we estimate $Var(Y_h)$. \\

Denote $\lfloor T_t/h\rfloor$ by $n$. Note that the variance does not change when added some constant, and denote $g\left(X_{t_i}\right)-g\left(X_0\right)=g_i$. We have
\begin{equation}
\begin{aligned}
Var Y_h(t,x)&=Var\left(\phi\left(X_{T_t}^x\right)+h\sum_{i=1}^ng\left(X_{t_i}^x\right)\right)\\
&=Var\left(\phi\left(X_{T_t}^x\right)+h\sum_{i=1}^n\left(g\left(X_{t_i}^x\right)-g\left(X_0\right)\right)\right)\\
&\leqslant\mathbb E\left(\phi\left(X_{T_t}^x\right)+h\sum_{i=1}^ng_i\right)^2\\
&\leqslant \mathbb E[\phi\left(X_{T_t}^x\right)^2]+h^2\mathbb E[\left(\sum_{i=1}^ng_i\right)^2]+2h\mathbb E[\phi\left(X_{T_t}^x\right)\sum_{i=1}^ng_i]
\end{aligned}
\end{equation}
Denote the upper bound of $\mathbb E\left[\phi\left(X_{T_t}\right)^2\right]$ by $M_1$. Next
\begin{equation}
\begin{aligned}
\mathbb E\left[g_i^2|T_t\right]&\leqslant L^2\mathbb E[|X_{t_i}-X_0|_\gamma^2|T_t]\\
&=L^2\mathbb E[\left(\sum_{j=1}^d|X_{t_i,(j)}-X_{0,(j)}|^\gamma\right)^2|T_t]\leqslant dL^2\mathbb E[\sum_{j=1}^d|X_{t_i,(j)}-X_{0,(j)}|^{2\gamma}|T_t]\\
&=dL^2t_i^{\frac{2\gamma}{\beta}}\mathbb E[\sum_{j=1}^d|X_{1,(j)}-X_{0,(j)}|^{2\gamma}].
\end{aligned}
\end{equation}

Use the result of (3.5), for $j=1,...,d,$
\begin{equation}
\mathbb E[|X_{1,(j)}-X_{0,(j)}|^{2\gamma}]=c^{2\gamma/\beta}\frac{2^{2\gamma}
\Gamma\left(\frac{1+2\gamma}{2}\right)\Gamma\left(1-\frac{2\gamma}{\beta}\right)}{\sqrt{\pi}\Gamma\left(1-\gamma\right)}.
\end{equation}
Let us denote
\begin{equation}
M_2:=\sum_{j=1}^d\mathbb E[|X_{1,(j)}-X_{0,(j)}|^{2\gamma}]
=dc^{2\gamma/\beta}\frac{2^{2\gamma}\Gamma\left(\frac{1+2\gamma}{2}\right)
\Gamma\left(1-\frac{2\gamma}{\beta}\right)}{\sqrt{\pi}\Gamma\left(1-\gamma\right)}.
\end{equation}
Then (3.9) becomes
\begin{equation}
\mathbb E[g_i^2|T_t]\leqslant dL^2t_i^{\frac{2\gamma}{\beta}}M_2,
\end{equation}
and for $i\neq j$,
\begin{equation}
\mathbb E\left[g_ig_j|T_t\right]\leqslant \left(\mathbb E[g_i^2|T_t]\mathbb E[g_j^2]|T_t\right)^{\frac{1}{2}}\leqslant dL^2t_i^{\frac{\gamma}{\beta}}t_j^{\frac{\gamma}{\beta}}M_2.
\end{equation}
Therefore,
\begin{equation}
\mathbb E[\left(\sum_{i=1}^ng_i\right)^2|T_t]\leqslant \sum_{i=1}^ndL^2t_i^{\frac{2\gamma}{\beta}}M_2+\sum_{i\neq j}2dL^2t_i^{\frac{\gamma}{\beta}}t_j^{\frac{\gamma}{\beta}}M_2
=dL^2M_2\left(\sum_{i=1}^nt_i^{\frac{\gamma}{\beta}}\right)^2.
\end{equation}
Recall that $t_i=ih$ and $n=\lfloor T_t/h\rfloor$, hence
\begin{equation}
\sum_{i=1}^nt_i^{\frac{\gamma}{\beta}}=h^{\frac{\gamma}{\beta}}\sum_{i=1}^ni^{\frac{\gamma}{\beta}}\leqslant h^{\frac{\gamma}{\beta}}\int_0^{n+1}x^{\frac{\gamma}{\beta}}dx=\frac{1}{1+\frac{\gamma}{\beta}}h^{\frac{\gamma}{\beta}}\left(n+1\right)^{1+\frac{\gamma}{\beta}}\leqslant \frac{1}{1+\frac{\gamma}{\beta}}h^{\frac{\gamma}{\beta}}\left(T_t/h+1\right)^{1+\frac{\gamma}{\beta}}.
\end{equation}
Hence
\begin{equation}
h^2\mathbb E[\left(\sum_{i=1}^ng_i\right)^2]=h^2\mathbb E[\mathbb E[\left(\sum_{i=1}^ng_i^2\right)|T_t]]\leqslant \frac{dL^2M_2}{\left(1+\frac{\alpha}{\beta}\right)^2}\mathbb E[\left(T_t+1\right)^{2\left(1+\frac{\gamma}{\beta}\right)}].
\end{equation}
Note that
\begin{equation}
\mathbb E[\left(T_t+1\right)^{2\left(1+\frac{\gamma}{\beta}\right)}]\leqslant\mathbb E[\left(T_t+1\right)^3]=\mathbb E[T_t^3]+3\mathbb E[T_t^2]+3\mathbb E[T_t]+1.
\end{equation}
And from (3.7), we know that for $k=1,2,3$
\begin{equation}
\mathbb E[T_t^k]=\bar a^{k\alpha}\mathbb E[\tau_1^{-k\alpha}]=\bar a^{k\alpha}\frac{\Gamma\left(1+k\right)}{\Gamma\left(1+k\alpha\right)}
\end{equation}
implying the upper bound of $h^2\mathbb E[\left(\sum_{i=1}^ng_i\right)^2]$. By Cauchy-Schwarz inequality,
\begin{equation}
h\mathbb E[\phi\left(X_{T_t}\right)\sum_{i=1}^ng_i]\leqslant \left(h^2\mathbb E[\left(\sum_{i=1}^ng_i\right)^2]\mathbb E[\left(\phi\left(X_{T_t}\right)\right)^2]\right)^{\frac{1}{2}}
\end{equation}
and hence we get the upper bound of $VarY_h(t,x)$ using (3.8):
\begin{equation}
VarY_h(t,x)\leqslant \mathbb E[\phi\left(X_{T_t}\right)^2]+\frac{dL^2M_2}{\left(1+\frac{\alpha}{\beta}\right)^2}\mathbb E[\left(T_t+1\right)^{2\left(1+\frac{\gamma}{\beta}\right)}]+\left(\mathbb E[\phi\left(X_{T_t}\right)^2]\frac{dL^2M_2}{\left(1+\frac{\alpha}{\beta}\right)^2}\mathbb E[\left(T_t+1\right)^{2\left(1+\frac{\gamma}{\beta}\right)}]\right)^{\frac{1}{2}}
\end{equation}
\begin{rem}
	 Using Example 2.7, we know the upper bound of $\mathbb E[\phi\left(X_{T_t}\right)^2]$ has order $d^{1+\frac{\beta}{2}}$. By (3.11), $M_2$ has order $d$. By (3.16), the upper bound of $h^2\mathbb E[\left(\sum_{i=1}^ng_i\right)^2]$ has order $d^2$. Hence the upper bound of $varY$  has order $d^2$.
\end{rem}
\subsection{Estimation of  EZ-EY}
Similarly, we begin with the estimation the conditional expectation.  \\

Conditioning on $T_t$ we write
\begin{equation}
\begin{aligned}
\mathbb E[Z(t,x)]-\mathbb E[Y_h(t,x)|T_t]=
&\mathbb E\left[\left(\phi\left(X_{T_t}\right)+\int_0^{T_t}g\left(X_s\right)ds-\left(\phi\left(X_{T_t}\right)+\sum_{i=1}^{\lfloor T_t/h\rfloor}hg\left(X_{t_i}\right)\right)\right)|T_t\right]\\
=&\mathbb E\left[\int_0^{T_t}g\left(X_s\right)ds-\sum_{i=1}^{\lfloor T_t/h\rfloor}hg\left(X_{t_i}\right)|T_t\right]\\
=&\mathbb E\left[\left(\sum_{i=1}^{\lfloor T_t/h\rfloor}\int_{t_i}^{t_i+h}\left(g\left(X_s\right)-g\left(X_{t_i}\right)\right)ds+\int_{\lfloor T_t/h\rfloor h}^{T_t}g\left(X_s\right)ds\right)|T_t\right].
\end{aligned}
\end{equation}
We have
\begin{equation}
\begin{aligned}
\mathbb E[|\sum_{i=1}^{\lfloor T_t/h\rfloor}\int_{t_i}^{t_i+h}\left(g\left(X_s\right)-g\left(X_{t_i}\right)\right)ds|\big|T_t]
&\leqslant
\mathbb E[\sum_{i=1}^{\lfloor T_t/h\rfloor}\int_{t_i}^{t_i+h}L|X_s-X_{t_i}|_\gamma ds\big|T_t]\\
\left(\text{ stationarity of increments}\right)&=
\mathbb E[\sum_{i=1}^{\lfloor T_t/h\rfloor}\int_0^hL|X_s-X_0|_\gamma ds\big|T_t]\\
&=
\mathbb E[\sum_{i=1}^{\lfloor T_t/h\rfloor}\int_0^hL\sum_{j=1}^d|X_{s,(j)}-X_{0,(j)}|^\gamma ds\big|T_t]\\
\left(\text{$X_t$ is $\beta$-stable}\right)&=\mathbb E[\lfloor T_t/h\rfloor L\int_0^h s^{\frac{\gamma}{\beta}}\sum_{j=1}^d|X_{1,(j)}-X_{0,(j)}|^\gamma ds|T_t]\\
&=C_0\lfloor T_t/h\rfloor h^{1+\frac{\gamma}{\beta}}\leqslant C_0T_th^{\frac{\gamma}{\beta}}
\end{aligned}
\end{equation}
where $C_0=\frac{1}{1+\frac{\gamma}{\beta}}L\sum_{j=1}^d\mathbb E|X_{1,(j)}-X_{0,(j)}|^\gamma$.\\
From (3.5), we know that
\begin{equation}
\mathbb E[|X_{1,(j)}-X_{0,(j)}|^\gamma]=c^{\frac{\gamma}{\beta}}\frac{2^\gamma\Gamma\left(\frac{1+\gamma}{2}\right)\Gamma\left(1-\frac{\gamma}{\beta}\right)}{\sqrt{\pi}\Gamma\left(1-\frac{\gamma}{2}\right)}.
\end{equation}
\\

Next
\begin{equation}
g\left(X_s\right)\leqslant g\left(X_{T_t}\right)+|g\left(X_{T_t}\right)-g\left(X_s\right)|,
\end{equation}
Thus
\begin{equation}
\begin{aligned}
\mathbb E g\left(X_s\right)&\leqslant \mathbb E\left(g\left(X_{T_t}\right)+|g\left(X_{T_t}\right)-g\left(X_s\right)|\right)\\
&\leqslant\mathbb E[g\left(X_{T_t}\right)+L|X_{T_t}-X_s|_\gamma]\\
&=\mathbb E[g\left(X_{T_t}\right)+L\left(T_t-s\right)^{\frac{\gamma}{\beta}}|X_1-X_0|],
\end{aligned}
\end{equation}
and
\begin{equation}
\begin{aligned}
\mathbb E[\int_{\lfloor T_t/h\rfloor h}^{T_t} g\left(X_s\right)ds|T_t]&\leqslant
\mathbb E[\int_{\lfloor T_t/h\rfloor h}^{T_t}\mathbb E[g\left(X_{T_t}\right)+\mathbb EL\left(T_t-s\right)^{\frac{\gamma}{\beta}}|X_1-X_0|_\gamma]ds|T_t]\\
&\leqslant
\mathbb E[h\mathbb E[g\left(X_{T_t}\right)+\frac{1}{1+\frac{\gamma}{\beta}}Lh^{1+\frac{\gamma}{\beta}}\mathbb E|X_1-X_0|_\gamma]|T_t].
\end{aligned}
\end{equation}
We have showed that if $g\left(x\right)=O\left(x^{\frac{\beta}{1+\delta}}\right)$, where $\delta>0$, then
$$
\mathbb E[g\left(X_{T_t}\right)]<M_3 <\infty.
$$
with some $M_3$.\\
Therefore, by (3.26) 
\begin{equation}
\mathbb E[\mathbb E[\int_{\lfloor T_t/h\rfloor h}^{T_t} g\left(X_s\right)ds|T_t]]\leqslant
M_2h+\frac{1}{1+\frac{\gamma}{\beta}}L\mathbb E|X_1-X_0|_\gamma h^{1+\frac{\gamma}{\beta}}.
\end{equation}
\\

Combining (3.21), (3.22) and (3.27) we have
\begin{equation}
\begin{aligned}
&|\mathbb E[Z-Y]|=\mathbb E[\mathbb E[Z-Y|T_t]]\\
\leqslant&
M_3h+Ldc^{\frac{\gamma}{\beta}}\frac{2^\gamma\Gamma\left(\frac{1+\gamma}{2}\right)\Gamma\left(1-\frac{\gamma}{\beta}\right)}{\left(1+\frac{\gamma}{\beta}\right)\sqrt{\pi}\Gamma\left(1-\frac{\gamma}{2}\right)}h^{1+\frac{\gamma}{\beta}}+Ldc^{\frac{\gamma}{\beta}}\frac{2^\gamma\Gamma\left(\frac{1+\gamma}{2}\right)\Gamma\left(1-\frac{\gamma}{\beta}\right)}{\left(1+\frac{\gamma}{\beta}\right)\sqrt{\pi}\Gamma\left(1-\frac{\gamma}{2}\right)}h^{\frac{\gamma}{\beta}}\bar a^\alpha\frac{\Gamma\left(2\right)}{\Gamma\left(1+\alpha\right)}\\
=&O\left(h^{\frac{\gamma}{\beta}}\right).
\end{aligned}
\end{equation}
\begin{rem}
\begin{enumerate}[1.]
	\item With similar argument, we can show that $\mathbb E[Z^2]<\infty$:\\ Since we have $\mathbb E[\phi\left(X_{T_t}^x\right)^2]<\infty$, we only need to prove $\mathbb E[\left(\int_0^{T_t}g\left(X_s\right)ds\right)^2]<\infty$. Like (3.21), we have
	\begin{equation}
	\begin{aligned}
	\mathbb E[\left(\int_0^{T_t}g\left(X_s\right)ds\right)^2]&\leqslant\mathbb E[\left(\int_0^{T_t}|g\left(X_s\right)|ds\right)^2]\\
	&\leqslant\mathbb E[\left(\int_0^{T_t}|g\left(X_0\right)|+|g\left(X_0\right)-g\left(X_s\right)|ds\right)^2]\\
	&=\mathbb E[\mathbb E[\left(\int_0^{T_t}|g\left(X_0\right)|+|g\left(X_0\right)-g\left(X_s\right)|ds\right)^2|T_t]]\\
	&\leqslant 2\mathbb E[\mathbb E[\left(\int_0^{T_t}|g\left(X_0\right)|ds\right)^2|T_t]+\mathbb E[\left(\int_0^{T_t}|g\left(X_0\right)-g\left(X_s\right)|ds\right)^2|T_t]].
	\end{aligned}
	\end{equation}
	Note that
	\begin{equation}
	\mathbb E[\mathbb E[\left(\int_0^{T_t}|g\left(X_0\right)|ds\right)^2|T_t]]=\mathbb E[T_t^2]|g\left(X_0\right)|^2<\infty,
	\end{equation}
	and
	\begin{equation}
	\begin{aligned}
		\mathbb E[\left(\int_0^{T_t}|g\left(X_0\right)-g\left(X_s\right)|ds\right)^2|T_t]&\leqslant \mathbb E[\left(\int_0^{T_t}L|X_0-X_s|_\gamma ds\right)^2|T_t]\\
		&=L^2\mathbb E[\frac{1}{T_t}\int_0^{T_t}|X_0-X_s|_\gamma^2ds|T_t]\\
		&=L^2\frac{1}{T_t}\mathbb E[\int_0^{T_t}\left(\sum_{j=1}^d|X_{0,(j)}-X_{s,(j)}|^\gamma\right)^2ds|T_t]\\
		&\leqslant L^2d\frac{1}{T_t}\mathbb E[\int_0^{T_t}\sum_{j=1}^d|X_{0,(j)}-X_{s,(j)}|^{2\gamma}ds|T_t]\\
		&\leqslant L^2d\frac{1}{T_t}\mathbb E[\int_0^{T_t}\sum_{j=1}^ds^{\frac{2\gamma}{\beta}}|X_{0,(j)}-X_{1,(j)}|^{2\gamma}ds|T_t]\\
		&=L^2d\frac{1}{T_t}\frac{1}{1+\frac{2\gamma}{\beta}}T_t^{1+\frac{2\gamma}{\beta}}\mathbb E[\sum_{j=1}^d|X_{0,(j)}-X_{1,(j)}|^{2\gamma}].
		\end{aligned}
	\end{equation}
	Hence
	\begin{equation}
	\mathbb E[\mathbb E[\left(\int_0^{T_t}|g\left(X_0\right)-g\left(X_s\right)|ds\right)^2|T_t]]\leqslant L^2d\frac{1}{1+\frac{2\gamma}{\beta}}\mathbb E[T_t^{\frac{2\gamma}{\beta}}]\mathbb E[\sum_{i=1}^d|X_{0,(j)}-X_{1,(j)}|^{2\gamma}]<\infty.
	\end{equation}
	Combining (3.30) and (3.32) we have
	\begin{equation}
	\mathbb E[\left(\int_0^{T_t}g\left(X_s\right)ds\right)^2]<\infty,
	\end{equation}
	and therefore $\mathbb E[Z^2]<\infty$.
	\item With the same condition, we can also show that $\mathbb E[|Y-Z|^2]$ has order $h^{\frac{2\gamma}{\beta}}$ using similar argument:
	\begin{equation}
	\mathbb E[|Y-Z|^2]=\mathbb E[\left(\sum_{i=1}^{\lfloor T_t/h\rfloor}\int_{\left(i-1\right)h}^{ih}(g\left(X_{s}\right)-g\left(X_{\left(i-1\right)h}\right))ds+\int_{\lfloor T_t/h\rfloor h}^{T_t}g\left(X_s\right)ds\right)^2].
	\end{equation}
And
	\begin{equation}
	\begin{aligned}
	\mathbb E[|\int_{\lfloor T_t/h\rfloor}^{T_t}g\left(X_s\right)ds|^2|T_t]&\leqslant\mathbb E[\left(T_t-\lfloor T_t/h\rfloor h\right)\int_{\lfloor T_t/h\rfloor h}^{T_t}g\left(X_s\right)^2ds]\\
	&\leqslant\mathbb E[h\int_{\lfloor T_t/h\rfloor h}^{T_t}\left(g\left(X_{T_t}\right)+|g\left(X_{T_t}\right)-g\left(X_s\right)|\right)^2ds]\\
	&\leqslant \mathbb E[2h\int_{\lfloor T_t/h\rfloor}^{T_t}g\left(X_{T_t}\right)^2]+2h\mathbb E[\int_{\lfloor T_t/h\rfloor h}^{T_t}\left(g\left(X_{T_t}\right)-g\left(X_s\right)\right)^2ds],
	\end{aligned}
	\end{equation}
	\begin{equation}
	\begin{aligned}
	\mathbb E[\int_{\lfloor T_t/h\rfloor h}^{T_t}\left(g\left(X_{T_t}\right)-g\left(X_s\right)\right)^2ds|T_t]&\leqslant L^2\mathbb E[\int_{\lfloor T_t/h\rfloor h}^{T_t}|X_{T_t}-X_s|_\gamma^2ds|T_t]\\
	&\leqslant L^2d\mathbb E[\int_{\lfloor T_t/h\rfloor h}^{T_t}\left(\sum_{j=1}^d|X_{T_t,(j)}-X_{s,(j)}|^{2\gamma}\right)ds|T_t]\\
	&=L^2d\mathbb E[\int_{\lfloor T_t/h\rfloor h}^{T_t}\left(T_t-s\right)^{\frac{2\gamma}{\beta}}\sum_{j=1}^d|X_{1,(j)}-X_{0,(j)}|^{2\gamma}ds|T_t]\\
	&\leqslant L^2d\mathbb E[\sum_{j=1}^d|X_{1,(j)}-X_{0,(j)}|^{2\gamma}\frac{1}{\frac{2\gamma}{\beta}+1}h^{1+\frac{2\gamma}{\beta}},
	\end{aligned}
	\end{equation}
	\begin{equation}
	\begin{aligned}
	&\mathbb E[\left(\sum_{i=1}^{\lfloor T_t/h\rfloor}\int_{\left(i-1\right)h}^{ih}g\left(X_{s}\right)-g\left(X_{\left(i-1\right)h}\right)ds\right)^2|T_t]\\
	\leqslant&\mathbb E[\lfloor T_t/h\rfloor\sum_{i=1}^{\lfloor T_t/h\rfloor}\left(\int_{\left(i-1\right)h}^{ih}g\left(X_s\right)-g\left(X_{\left(i-1\right)h}\right)ds\right)^2|T_t]\\
	\leqslant&\lfloor T_t/h\rfloor L^2\sum_{i=1}^{\lfloor T_t/h\rfloor}\mathbb E[\left(\int_{\left(i-1\right)h}^{ih}|X_s-X_{\left(i-1\right)h}|_\gamma ds\right)^2|T_t]\\
	\leqslant&\lfloor T_t/h\rfloor L^2\sum_{i=1}^{\lfloor T_t/h\rfloor}\mathbb E[h\int_{\left(i-1\right)h}^{ih}|X_s-X_{\left(i-1\right)h}|_\gamma^2 ds|T_t]\\
	=&\lfloor T_t/h\rfloor L^2\sum_{i=1}^{\lfloor T_t/h\rfloor}\mathbb E[h\int_{\left(i-1\right)h}^{ih}\left(\sum_{j=1}^d|X_{s,(j)}-X_{\left(i-1\right)h,(j)}|^\gamma\right)^2 ds|T_t]\\
	\leqslant&\lfloor T_t/h\rfloor L^2\sum_{i=1}^{\lfloor T_t/h\rfloor}\mathbb E[h\int_{\left(i-1\right)h}^{ih}d\sum_{j=1}^d|X_{s,(j)}-X_{\left(i-1\right)h,(j)}|^{2\gamma} ds|T_t]\\
	=&\lfloor T_t/h\rfloor^2 L^2\mathbb E[h\int_0^hs^{\frac{2\gamma}{\beta}}\sum_{j=1}^d|X_{1,(j)}-X_{0,(j)}|^{2\gamma} ds|T_t]\\
	\leqslant&T_t^2L^2\mathbb E[\sum_{j=1}^d|X_{1,(j)}-X_{0,(j)}|^{2\gamma}]h^{\frac{2\gamma}{\beta}}.
	\end{aligned}
	\end{equation}
Hence $ \mathbb E[|Y-Z|^2]\leqslant M_{t,x}^3h^{\frac{2\gamma}{\beta}}$ where $M_{t,x}^3$ is a constant that only depends on $t$ and $x$.
	
\end{enumerate}
\end{rem}
\subsection{Proof of Central Limit Theorem with a bias correction}
Recall the definition of \textbf{null array}. By this we mean a triangular array of random variables $\left(\xi_{nj}\right), 1\leqslant j\leqslant m_n, n, m_n\in\mathbb N$, such that the  $\xi_{nj}$ are independent for each $n$ and satisfy
\begin{equation}
\sup_j\mathbb E[|\xi_{nj}|\land 1]\to 0.
\end{equation}
The following result is well-known (see \cite{kallenberg2002foundations} theorem 5.15).
\begin{thm}
Let $\left(\xi_{nj}\right)$ be a null array of random variables, then $\sum_{j=1}^{m_n}\xi_{nj}\overset{d}\to N\left(b,c\right)$ iff these conditions hold:
\begin{enumerate}[(i)]
	\item $\sum_{j=1}^{m_n}\mathbb P[|\xi_{nj}|>\epsilon]\to0$ for all $\epsilon>0$ as $n\to\infty$;
	\item $\sum_{j=1}^{m_n}\mathbb E[\xi_{nj};|\xi_{nj}|\leqslant 1]\to b$ as $n\to\infty$, where $\mathbb E[X;A]=\mathbb E[X\mathbb I_A]$,
	\item $\sum_{j=1}^{m_n}Var[\xi_{nj};|\xi_{nj}|\leqslant 1]\to c$ as $n\to\infty$, where $Var\left(X;A\right):=Var\left(X\mathbb I_A\right)$.
\end{enumerate}
\end{thm}
Again denote $Z:=\phi\left(X_{T_t}^x\right)+\int_0^{T_t}g\left(X_s^x\right)ds$, $Y_h^k:=\phi\left(X_{T_t^k}^x\right)+\sum_{i=1}^{\lfloor T_t^k/h\rfloor}hg\left(X_{ih}^x\right)$ where $T_t^k$ are independent samples of the stopping time. Now we will apply Theorem 3.8. to prove Theorem 3.1.

\begin{proof}[\textbf{Proof of Theorem 3.4}]
Let $\xi_{Nj}=\frac{1}{\sqrt{N}}\left(Y_{h_N}^j-\mathbb EZ\right)$, $j=1,...,N$, then for any $\epsilon>0$,
\begin{equation}
\begin{aligned}
\mathbb P[|\xi_{Nj}|>\epsilon]&=\mathbb P[|Y_{h_N}^j-\mathbb EZ|>\sqrt{N}\epsilon]\leqslant \frac{\mathbb E[|Y_{h_N}^j-\mathbb EZ|]}{\sqrt N\epsilon}\\
&\leqslant \frac{\mathbb E[|Y_{h_N}^j-\mathbb E[Y_{h_N}^j]|]+\mathbb E[|Y_{h_N}^j-Z|]}{\sqrt N\epsilon}\\
&\leqslant \frac{var\left(Y_{h_N}^j\right)^\frac{1}{2}+\mathbb E[|Y_{h_N}^j-Z|]}{\sqrt N\epsilon}
\leqslant\frac{\sqrt{M_{t,x}^1}+M_{t,x}^2h_N^{\frac{\gamma}{\beta}}}{\sqrt N\epsilon}\to0\qquad\text{as}\quad N\to\infty,
\end{aligned}
\end{equation}
where $M_{t,x}^1,M_{t,x}^2$ are the same as above.
This implies that $\xi_{Nj}$ converges to 0 in probability uniformly in $N$, and therefore $\left(\xi_{Nj}\right)$ is a null array.\\
Denote $A_{Nj}=\{|Y_{h_N}^j-\mathbb EZ|\leqslant\sqrt N\}=\{|\xi_{Nj}|\leqslant 1\}$. To apply Theorem 3.8, we only need to check that those three conditions hold.
\begin{enumerate}[(i)]
	\item
	We need to prove that for any $\epsilon>0$,
	\begin{equation}
	\sum_{j=1}^N\mathbb P[|\xi_{Nj}|>\epsilon]=\sum_{j=1}^N\mathbb P[|Y_{h_N}^j-\mathbb EZ|>\sqrt N\epsilon]\to0.
	\end{equation}
	Note that
	$$
	\{|Y_{h_N}^j-\mathbb EZ|>\sqrt{N}\epsilon\}\subset\{|Y_{h_N}^j-Z|>\frac{1}{2}\sqrt{N}\epsilon\}\cup\{|Z-\mathbb EZ|>\frac{1}{2}\sqrt{N}\epsilon\}.
	$$
	Hence
	\begin{equation}
	\mathbb P[|Y_{h_N}^j-\mathbb EZ|>\sqrt{N}\epsilon]\leqslant \mathbb P[|Y_{h_N}^j-Z|>\frac{1}{2}\sqrt{N}\epsilon]+\mathbb P[|Z-\mathbb EZ|>\frac{1}{2}\sqrt{N}\epsilon],
	\end{equation}
	\begin{equation}
	\sum_{j=1}^N\mathbb P[|Y_{h_N}^j-\mathbb EZ|>\sqrt{N}\epsilon]
	\leqslant\sum_{j=1}^n\mathbb P[|Y_{h_N}^j-Z|>\frac{1}{2}\sqrt{N}\epsilon]+\sum_{j=1}^n\mathbb P[|Z-\mathbb EZ|>\frac{1}{2}\sqrt{N}\epsilon],
	\end{equation}
	\begin{equation}
	\begin{aligned}
	\sum_{j=1}^N\mathbb P[|Z-\mathbb EZ|>\frac{1}{2}\sqrt{N}\epsilon]&=\sum_{j=1}^N\mathbb E[1;|Z-\mathbb EZ|>\frac{1}{2}\sqrt{N}\epsilon]\\
	&\leqslant\sum_{j=1}^N\mathbb E[\frac{4|Z-\mathbb EZ|^2}{N\epsilon^2};|Z-\mathbb EZ|>\frac{1}{2}\sqrt{N}\epsilon]\\
	&=\mathbb E[\frac{4|Z-\mathbb EZ|^2}{\epsilon^2};|Z-\mathbb EZ|>\frac{1}{2}\sqrt{N}\epsilon]\to0 \text{ as } N\to\infty,
	\end{aligned}
	\end{equation}
	\begin{equation}
	\sum_{j=1}^N\mathbb P[|Y_{h_N}^j-Z|>\frac{1}{2}\sqrt{N}\epsilon]\leqslant\sum_{j=1}^N\mathbb E[|Y_{h_N}^j-Z|]/(\frac{1}{2}\sqrt{N}\epsilon)\leqslant2\sqrt NM_{t,x}^1h_N^{\frac{\gamma}{\beta}}/\epsilon\to0 \text{ as } N\to\infty.
	\end{equation}
	Together with (3.43) and (3.44) we know that (3.40) holds.
	
	\item We need to prove that
	\begin{equation}
	\sum_{j=1}^N\mathbb E[\xi_{Nj};|\xi_{Nj}|\leqslant1]
=\frac{1}{\sqrt N}\sum_{j=1}^N\mathbb E[Y_{h_N}^j-\mathbb EZ;A_{Nj}]\to0\text{ as }N\to\infty.
	\end{equation}
	Since
	\begin{equation}
	\begin{aligned}
	\frac{1}{\sqrt N}\sum_{j=1}^N\mathbb E[Y_{h_N}^j-\mathbb EZ]=\frac{1}{\sqrt N}\sum_{j=1}^N\mathbb E[Y_{h_N}^j-Z]\leqslant\frac{1}{\sqrt N}\sum_{j=1}^N\mathbb E[|Y_{h_N}^j-Z|]\\
	\leqslant\frac{1}{\sqrt N}\sum_{j=1}^NM_{t,x}^2h_N^{\frac{\gamma}{\beta}}\to0\text{ as }N\to\infty,
	\end{aligned}
	\end{equation}
	we only need to prove
	\begin{equation}
	\frac{1}{\sqrt N}\sum_{j=1}^N\mathbb E[Y_{h_N}^j-\mathbb EZ;A_{Nj}^C]\to 0 \text{ as } N\to\infty.
	\end{equation}
	For a random variable $X$, we denote $X_+=\max\{X,0\}$.
	Then
	\begin{equation}
	\begin{aligned}
	\frac{\sqrt N}{N}\sum_{j=1}^N\mathbb E[\left(Y_{h_N}^j-\mathbb EZ\right)_+;A_{Nj}^C]&=\sqrt{N}\mathbb E[\left(Y_{h_N}^1-\mathbb EZ\right)_+;A_{N1}^C]\\
	(\text{note that } A_{N1}^C=\{|Y_{h_N}^1-\mathbb EZ|>\sqrt N\})&\leqslant\mathbb E[\left(Y_{h_N}^1-\mathbb EZ\right)_+^2;A_{N1}^C]\\
	&\leqslant\mathbb E[2\left(Y_{h_N}^1-Z\right)^2+2\left(Z-\mathbb EZ\right)^2;A_{N1}^C].
	\end{aligned}
	\end{equation}
	By Proposition 3.4,
	\begin{equation}
	\mathbb E[\left(Y_{h_N}^1-Z\right)^2]\leqslant M_{t,x}^3h_N^{\frac{2\gamma}{\beta}}\to0\text{ as }N\to\infty.
	\end{equation}
	Since $\mathbb P[A_{N1}]\to0$ as $N\to\infty$, we have
	\begin{equation}
	\mathbb E[\left(Z-\mathbb EZ\right)^2;A_{N1}^C]\to0\text{ as }N\to\infty.
	\end{equation}
	Together with (3.49) and (3.50) we have
	\begin{equation}
	\frac{1}{\sqrt N}\sum_{j=1}^N\mathbb E[\left(Y_{h_N}^j-\mathbb EZ\right)_+;A_{Nj}^C]\to 0.
	\end{equation}
	Similarly, denote $X_-:=\max\{-X,0\}$. Then
	\begin{equation}
	\frac{1}{\sqrt N}\sum_{j=1}^N\mathbb E[\left(Y_{h_N}^j-\mathbb EZ\right)_-;A_{Nj}^C]\to 0.
	\end{equation}
	Combining (3.51) and (3.52) we get (3.47) holds.
	\item We want to prove
	\begin{equation}
	\sum_{j=1}^NVar[\xi_{Nj};|\xi_{Nj}|\leqslant1]=\sum_{j=1}^N\frac{1}{N}Var[Y_{h_N}^j-\mathbb EZ;A_{Nj}]\to var\left(Z\right)\text{ as }N\to\infty.
	\end{equation}
	 We have following equality
	\begin{equation}
	Var\left(Y_{h_N}^j-\mathbb EZ;A_{Nj}\right)=\mathbb E[|Y_{h_N}^j-\mathbb EZ|^2\mathbb I_{A_{Nj}}]-\left(\mathbb E[\left(Y_{h_N}^j-\mathbb EZ\right)\mathbb I_{A_{Nj}}]\right)^2.
	\end{equation}
	Note that
	\begin{equation}
	\begin{aligned}
	\frac{1}{N}\sum_{j=1}^N\left(\mathbb E[\left(Y_{h_N}^j-\mathbb EZ\right)\mathbb I_{A_{Nj}}]\right)^2&=\frac{1}{N}\sum_{j=1}^N\left(\mathbb E[\left(Y_{h_N}^j-Z+Z-\mathbb EZ\right)\mathbb I_{A_{Nj}}]\right)^2\\
	&\leqslant\frac{2}{N}\sum_{j=1}^N\left(\mathbb E[\left(Y_{h_N}^j-Z\right)\mathbb I_{A_{Nj}}]^2+\mathbb E[\left(Z-\mathbb EZ\right)\mathbb I_{A_{Nj}}]^2\right).
	\end{aligned}
	\end{equation}
	Since
	\begin{equation}
	\mathbb E[\left(Y_{h_N}^j-Z\right)\mathbb I_{A_{Nj}}]\leqslant\mathbb E[|Y_{h_N}^j-Z|]\leqslant M_{t,x}^2h_N^{\frac{\gamma}{\beta}},
	\end{equation}
	We have
	\begin{equation}
	\frac{1}{N}\sum_{j=1}^N\mathbb E[\left(Y_{h_N}^j-Z\right)\mathbb I_{A_{Nj}}]^2\to0\text{ as }N\to\infty.
	\end{equation}
	Note that $0-\mathbb E[(Z-\mathbb EZ)\mathbb I_{A_{N1}}]=\mathbb E[(Z-\mathbb EZ)\mathbb I_{A_{N1}^C}]$ and $\mathbb P[A_{N1}^C]\to0$ as $N\to\infty$, hence
	\begin{equation}
	\frac{2}{N}\sum_{j=1}^N\mathbb E[\left(Z-\mathbb EZ\right)\mathbb I_{A_{Nj}}]^2=2\mathbb E[\left(Z-\mathbb EZ\right)\mathbb I_{A_{N1}^C}]^2\to0\text{ as }N\to\infty.
	\end{equation}
	
	Together with (3.57) and (3.58) we know the right hand side of (3.55) converges to 0, and therefore from (3.54) we only need to prove
	\begin{equation}
	\frac{1}{N}\sum_{j=1}^N\mathbb E[|Y_{h_N}^j-\mathbb EZ|^2\mathbb I_{A_{Nj}}]=\mathbb E[|Y_{h_N}^1-\mathbb EZ|^2\mathbb I_{A_{N1}}]\to Var\left(Z\right).
	\end{equation}
	Note that
	\begin{equation}
	\begin{aligned}
	\mathbb E[\left(Y_{h_N}^1-\mathbb EZ\right)^2]-\mathbb E[\left(Z-\mathbb EZ\right)^2]
	&=\mathbb E[(Y_{h_N}^1+Z-2\mathbb EZ)\left(Y_{h_N}^1-Z\right)]\\
	&\leqslant(\mathbb E[(Y_{h_N}^1+Z-2\mathbb EZ)^2]\mathbb E[(Y_{h_N}^1-Z)^2])^{\frac{1}{2}}.
	\end{aligned}
	\end{equation}
	Since $\mathbb E[Z^2]<\infty$ and $\mathbb E[|Y_{h_N}-Z|^2]<M_{t,x}^3h^{2\frac{\gamma}{\beta}}$, we know $E[|Y_{h_N}|^2]$ have a uniform upper bound for all $N$. Hence $\mathbb E[(Y_{h_N}^1+Z-2\mathbb EZ)^2]$ have a uniform upper bound for all $N$. Hence the right hand side of (3.60) converges to $0$ as $N\to\infty$. \\
	Therefore we only need to show
	\begin{equation}
	\mathbb E[|Y_{h_N}^1-\mathbb EZ|^2\mathbb I_{A_{N1}^C}]\to0 \text{ as }N\to\infty.
	\end{equation}
	In fact, this is true because
	\begin{equation}
	\mathbb E[|Y_{h_N}^1-\mathbb EZ|^2\mathbb I_{A_{N1}^C}]\leqslant2\mathbb E[|Y_{h_N}^1-Z|^2\mathbb I_{A_{N1}^C}]+2\mathbb E[|Z-\mathbb EZ|^2\mathbb I_{A_{N1}^C}].
	\end{equation}
\end{enumerate}
\end{proof}
\begin{rem}
The choice of $h_k$ is not unique. In fact, they only need to satisfy $\sqrt{N}h_N^{\frac{\gamma}{\beta}}\to0$ as $N\to\infty$.
\end{rem}
\section{Simulation and algorithm}
Now we study how the starting level $t$ of the decreasing subordinator and the starting point $x$ of the stable process $X$ influence the Monte Carlo estimator (1.4).\\

We set $d=1$, $\alpha=1/2$, $\beta=3/2$,  and denote $\bar a=t-a$. In section 4.1 and 4.2 we set $\phi\left(x\right)=|x|^{\frac{1}{2}}$.
\subsection{Unbiased FPED}
Now the estimator is (1.5). Recall that $X_{T_t}^x\overset{d}=T_t^\frac{1}{\beta}X_1+x\overset{d}=\left(\frac{\bar a}{\tau_1}\right)^{\frac{\alpha}{\beta}}X_1+x$.
\begin{algorithm}[H]
\caption{Sample $u_N\left(t,x\right)$}
\begin{algorithmic}[1]
\STATE $u=0$;
\FOR{$k=1:N$}
\STATE sample $Y_1$;
\STATE $T_t=\left(\frac{\bar a}{\tau_1}\right)^\alpha$;
\STATE sample $X_1$;
\STATE $X=T_t^{\frac{1}{\beta}}X_1+x$;
\STATE $u=u+\phi\left(X\right)$;
\ENDFOR
\STATE $\bar u=u/N$;
\RETURN $\bar u$.
\end{algorithmic}
\end{algorithm}
First we set $N=10^5$. Let $x=0$ and $\bar a$ increase from $1$ to $10$.

\begin{figure}[H]
    \centering
    \subfigure[$u_N$ when $g=0,x=0$]{
    \includegraphics[width=0.45\textwidth]{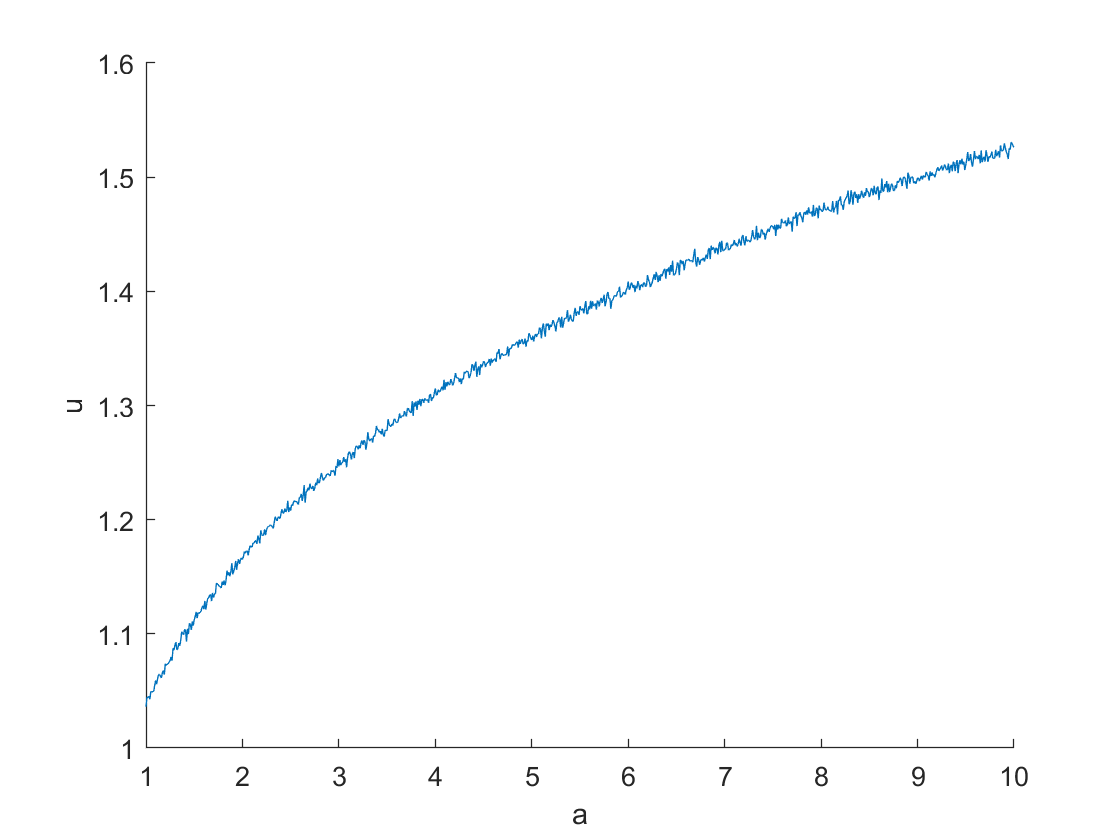}
}
    \subfigure[$u_N/{\bar a}^{\frac{1}{6}}$]{
    \includegraphics[width=0.45\textwidth]{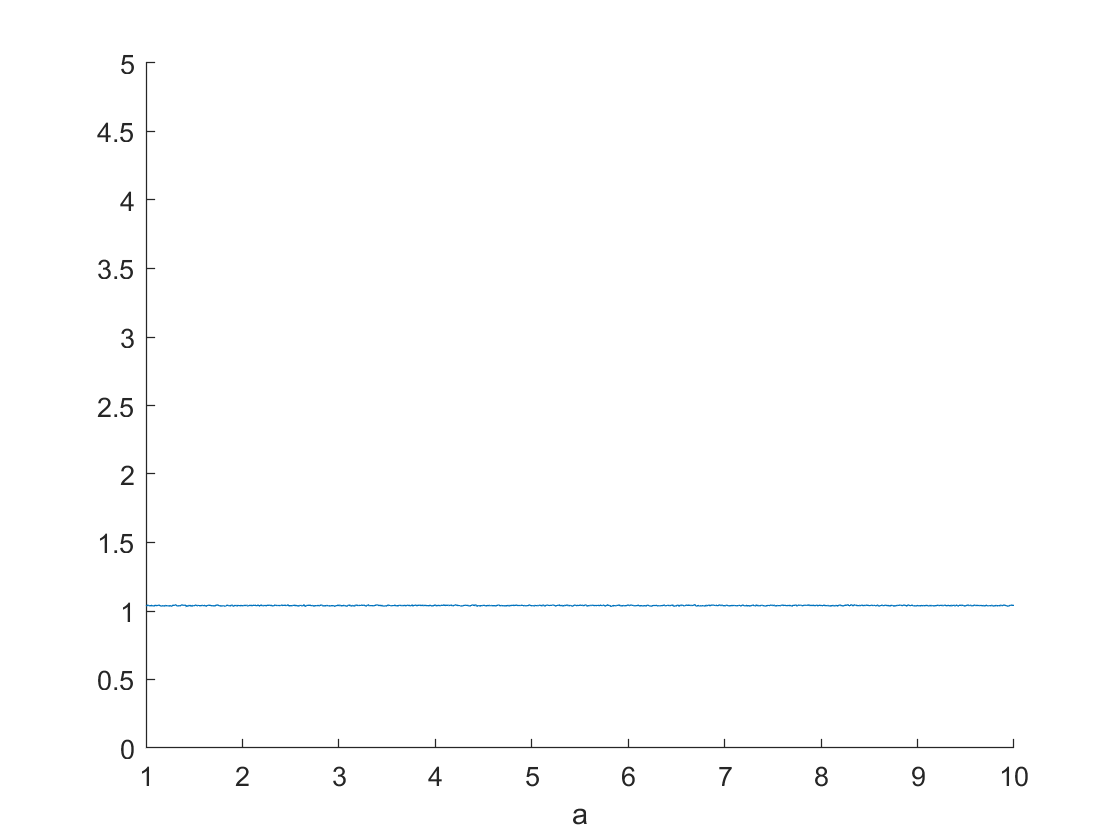}
}
    \caption{.}
 \end{figure}

In fact, now we have
\begin{equation}
\mathbb E[|X_{T_t}^x|^{\frac{1}{2}}]=\mathbb E[\left(\frac{\bar a}{\tau_1}\right)^{\frac{\alpha}{2\beta}}|X_1|^{\frac{1}{2}}]=\bar a^{\frac{\alpha}{2\beta}}\mathbb E[\tau_1^{-\frac{\alpha}{2\beta}}]\mathbb E[|X_1|^{\frac{1}{2}}]
\end{equation}
We can check our result by the right figure (b) `$u_N/\left(\bar a\right)^{\frac{\alpha}{\beta}}$' above. It is almost a constant, which means our algorithm is correct.
\subsection{FPDE with bias}
We set $g\left(x\right)=|x|^{\frac{1}{2}}$.
\begin{algorithm}[H]
\caption{Sample $u_N^h\left(t,x\right)$}
\begin{algorithmic}[1]
\STATE $u=0$;
\FOR{$k=1:N$}
\STATE sample $Y_1$;
\STATE $T_t=\left(\frac{\bar a}{\tau_1}\right)^\alpha$;
\STATE sample $X_1^j$, $j=1,...,\lfloor T_t/h\rfloor$;

\STATE $S=0$;
\STATE $X=x$;
\STATE sample $X_1'$;
\FOR{$j=1:\lfloor T_t/h\rfloor$}
\STATE $X=X+h^{\frac{1}{\beta}}X_1^j$;
\STATE $S=S+hg\left(X\right)$;
\ENDFOR
\STATE $X=X+\left(T_t-h\lfloor T_t/h\rfloor\right)^{\frac{1}{\beta}}X_1'$;
\STATE $u=u+\phi\left(X\right)+S$;
\ENDFOR
\STATE $\bar u=u/N$;
\RETURN $\bar u$.
\end{algorithmic}
\end{algorithm}
Figure 4.2 is the figure of $u_N^h$ when $x=0, h=0.01, N=10^5$ and we change $a$ from $1$ to $10$.
\begin{figure}[H]
 \centering
 \includegraphics[width=0.5\textwidth]{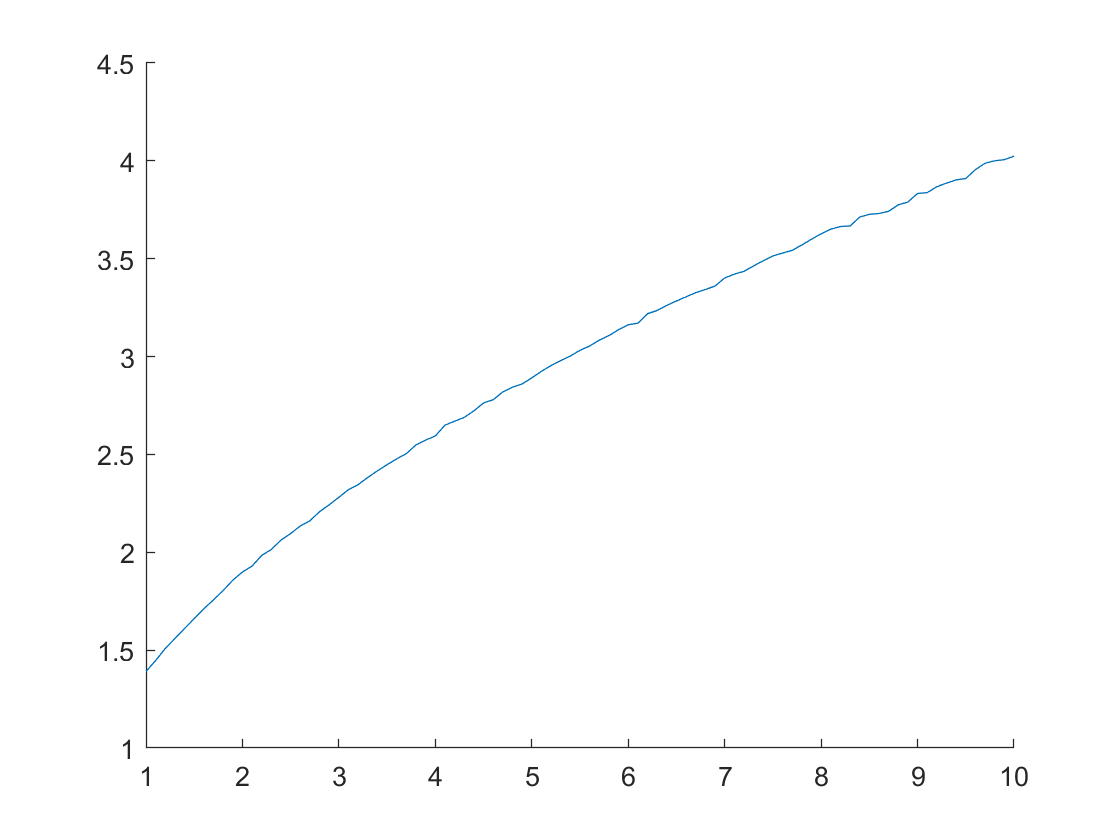}
 \caption{$x=0, \bar a=1:0.1:10$}

\end{figure}
Similarly, recall that $T_t\overset{d}= \left(\frac{\bar a}{\tau_1}\right)^\alpha$,
\begin{equation}
\begin{aligned}
\mathbb E[\int_0^{T_t}g\left(X^x_s\right)ds]&=\mathbb E[\mathbb E[\int_0^{T_t}|X^0_s|^{\frac{1}{2}}ds|T_t]]\\
&=\mathbb E[\mathbb E[\int_0^{T_t}|X^0_1|^{\frac{1}{2}}s^{\frac{1}{2\beta}}ds|T_t]]\\
&=\mathbb E[|X^0_1|^{\frac{1}{2}}]\mathbb E[T_t^{1+\frac{1}{2\beta}}]/\left(1+\frac{1}{2\beta}\right)\\
&=\bar a^{\frac{\alpha\left(1+2\beta\right)}{2\beta}}\mathbb E[|X^0_1|^{\frac{1}{2}}]\mathbb E[\tau_1^{-\frac{\alpha\left(1+2\beta\right)}{2\beta}}]/\left(1+\frac{1}{2\beta}\right).
\end{aligned}
\end{equation}
We can check our result by Figure 4.3, which is almost a constant.
\begin{figure}[H]
 \centering
 \includegraphics[width=0.5\textwidth]{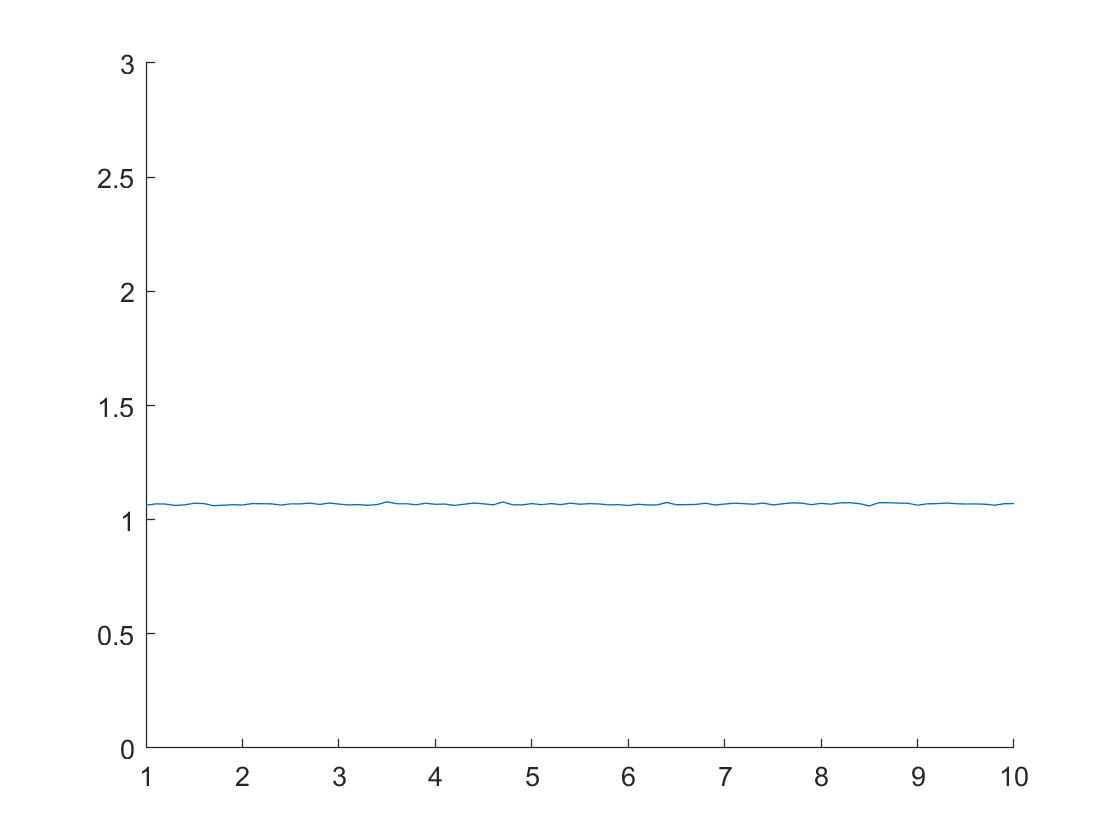}
 \caption{We set $x=0, \bar a=1:0.1:10$. As we can see, $\frac{1}{N}\sum_{k=1}^N\sum_{i=1}^{\lfloor T_t^k/h\rfloor}hg(X_{t_i^k}^k)/(\bar a)^{2/3}$ is almost a constant, which is consistent with our calculation in (4.2)}

\end{figure}
Below is the figure of $u_N^h$ when we fix $\bar a=5$, and $x=0:0.1:10$.
\begin{figure}[H]
 \centering
 \includegraphics[width=0.5\textwidth]{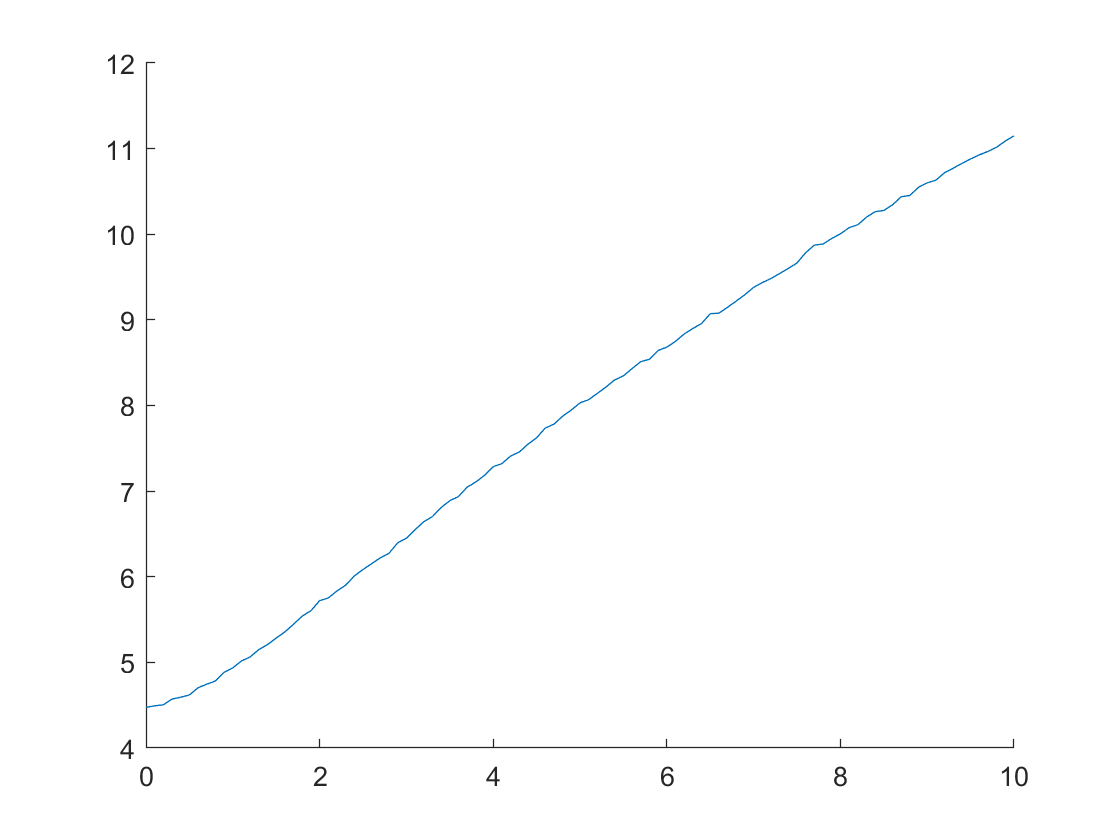}
 \caption{$a=5, x=0:0.1:10$}
\end{figure}
\subsection{Confidence interval}
For simplicity, we set $\phi\left(x\right)\equiv1$, $g\left(x\right)=|x|^{\frac{1}{2}}$, $\bar a=t-a=1$, $x=0$, $h=10^{-3}$. Recall our discussion in Remark 3.5. We only need the upper bound of $\mathbb E\left[Z(t,x)^2\right]$ in this example. In fact, in Remark 3.7.1, we have already got the computable upper bound of $\mathbb E\left[Z(t,x)^2\right]$. Figure 4.5 presents the asymptotic confidence intervals at level $95\%$.
\begin{figure}[H]
 \centering
 \includegraphics[width=0.5\textwidth]{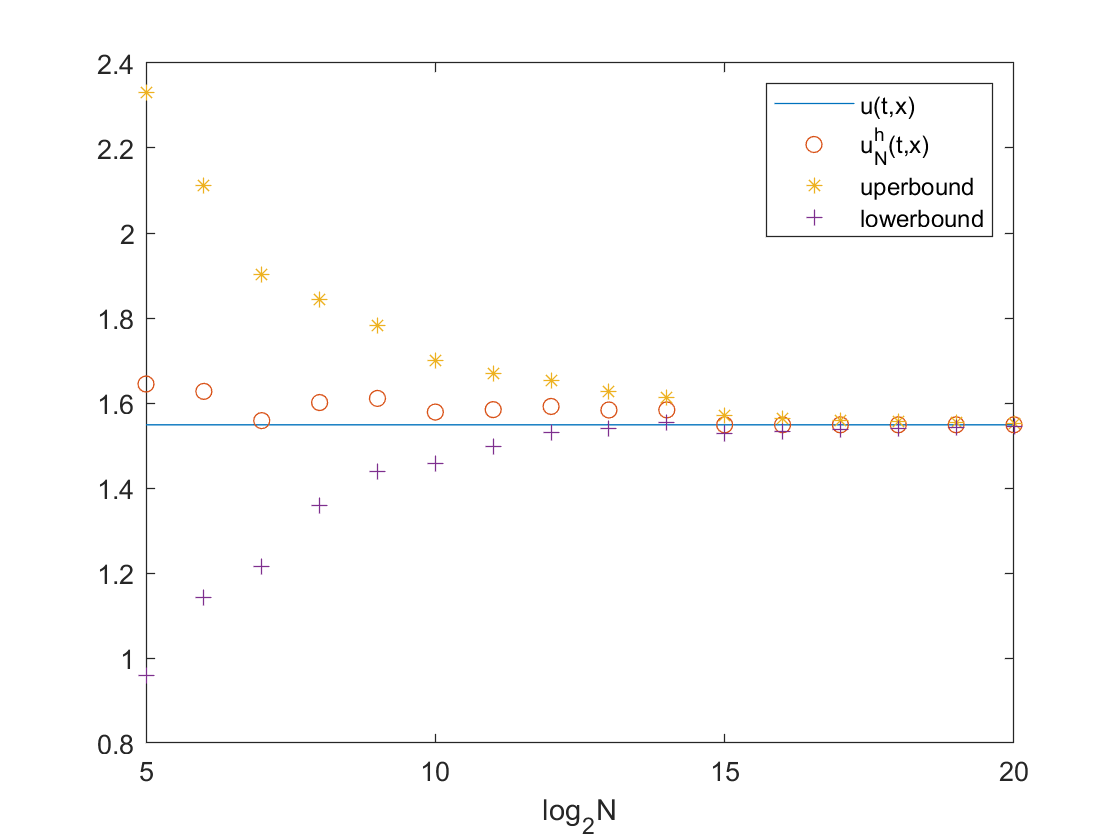}
 \caption{Confidence interval at level $95\%$}

\end{figure}
\bibliographystyle{plain}
\bibliography{refs}

\section*{Acknowledgement}
\noindent
VK is supported by the Russian Science Foundation
project no. 20-11-20119.
FL is supported by the China Scholarship Council PhD award at Warwick.
AM is supported by The Alan Turing Institute under the EPSRC grant EP/N510129/1 and
by the EPSRC grant EP/P003818/1 and the Turing Fellowship funded by the Programme on Data-Centric Engineering of Lloyd's Register Foundation;

\end{document}